\documentclass[11pt, oneside]{article}
\usepackage{amsfonts}
\usepackage{mathrsfs}
\usepackage{color}
\usepackage[colorlinks]{hyperref}
\usepackage{latexsym}
\usepackage{amssymb}
\usepackage{amsmath}
\usepackage{enumerate}
\usepackage{amsthm}
\usepackage{indentfirst}
\usepackage{mathtools}
\usepackage{tikz}
\usepackage{psfrag}
\usepackage{graphicx}
\usepackage{epstopdf}
\usepackage{float}
\usepackage{subfig}
\usepackage{bm}
\usepackage[rflt]{floatflt}
\usepackage{float}
\usepackage{authblk}
\usepackage{ulem}
\usepackage[square, comma, sort&compress, numbers]{natbib}
\usepackage{verbatim}
\usepackage{amsthm}
\usepackage{esint}

\numberwithin{equation}{section}
\allowdisplaybreaks

\newenvironment{proof2.1}{\medskip\noindent{\bf Proof of the Theorem 2.1:}\enspace}{\hfill \qed \newline \medskip}

\newenvironment{proof2.2}{\medskip\noindent{\bf Proof of the Theorem 2.2:}\enspace}{\hfill \qed \newline \medskip}

\newtheorem{theorem}{\color{black}\indent Theorem}[section]
\newtheorem{lemma}{\color{black}\indent Lemma}[section]

\newtheorem{remark}{\color{black}\indent Remark}[section]

\usepackage{amssymb,amsmath}
\begin{document}
\title{Mountain pass solution to the Br\'{e}zis-Nirenberg problem with logarithmic perturbation}
\author{Qian Zhang\qquad Yuzhu Han$^{\dag}$}

\affil{School of Mathematics, Jilin University,
 Changchun 130012, P.R. China}
\renewcommand*{\Affilfont}{\small\it}
\date{} \maketitle
\vspace{-20pt}

\footnotetext{\hspace{-1.9mm}$^\dag$Corresponding author.\\
Email addresses: yzhan@jlu.edu.cn(Y. Han).

\thanks{
$^*$Supported by the National Key Research and Development Program of China
(grant no.2020YFA0714101).}}

{\bf Abstract}
In this paper we give a positive answer to the conjecture raised by Hajaiej et al. (J. Geom. Anal., 2024,
34(6): No. 182, 44 pp) on the existence of a mountain pass solution at positive energy level to the Br\'{e}zis-Nirenberg problem with logarithmic perturbation. To be a little more precise, by taking full
advantage of the local minimum solution and some very delicate estimates on the logarithmic term and the
critical term, we prove that the following problem
\begin{eqnarray*}
\begin{cases}
-\Delta u= \lambda u+\mu|u|^2u+\theta u\log u^2, &x\in\Omega,\\
u=0, &x\in\partial\Omega
\end{cases}
\end{eqnarray*}
possesses a positive mountain pass solution at positive energy level, where $\Omega\subset \mathbb{R}^4$
is a bounded domain with smooth boundary $\partial\Omega$, $\lambda\in \mathbb{R}$, $\mu>0$ and $\theta<0$.
A key step in the proof is to control the mountain pass level around the local minimum solution from above
by a proper constant to ensure the local compactness. Moreover, this result is also extended to
three-dimensional and five-dimensional cases.

{\bf Keywords} Br\'{e}zis-Nirenberg problem; critical exponent; logarithmic perturbation;
mountain pass solution.

{\bf AMS Mathematics Subject Classification 2020:} Primary 35J20; Secondary 35J60.

\section{Introduction}

Consider the following Br\'{e}zis-Nirenberg problem with logarithmic perturbation
\begin{eqnarray}\label{P1}
\begin{cases}
-\Delta u= \lambda u+\mu|u|^2u+\theta u\log u^2, &x\in\Omega,\\
u=0, &x\in\partial\Omega,
\end{cases}
\end{eqnarray}
where $\Omega\subset \mathbb{R}^4$ is a bounded domain with smooth boundary
$\partial\Omega$, $\lambda\in \mathbb{R}$, $\mu>0$ and $\theta<0$.
Equation \eqref{P1} is closely related to the following time-dependent nonlinear logarithmic type
Schr\"{o}dinger equation
\begin{eqnarray}\label{P1-1}
\begin{cases}
\iota \partial_t\Psi=\Delta\Psi+\mu|\Psi|^2\Psi+\theta \Psi\log \Psi^2, \ \ x\in\Omega, \ t>0,\\
\Psi=\Psi(x,t)\in \mathbb{C},\\
\Psi(x,t)=0, \ \ x\in\partial\Omega, \ t>0,
\end{cases}
\end{eqnarray}
where $\iota$ is the imaginary unit. Equation \eqref{P1-1} appears in many physical fields, such as
quantum mechanics, quantum optics, nuclear physics, theory of superfluidity and Bose-Einstein condensation,
see e.g., \cite{BB1975,Carles2018,Choi1999,Erdos2010}. Consequently, the research on nonlinear logarithmic
type Schr\"{o}dinger equations  has attracted extensive attention during the past few decades.

In particular, Deng et al. \cite{Dengyinbin} investigated the following Br\'{e}zis-Nirenberg problem with
logarithmic perturbation
\begin{equation}\label{P3}
\begin{cases}
-\Delta{u}=\lambda u+\mu|u|^{2^{*}-2}u+\theta u\log{u^2},&x\in\Omega,\\
u=0,&x\in\partial\Omega,
\end{cases}
\end{equation}
where $\Omega\subset\mathbb{R}^N~(N\geq3)$ is a bounded domain with smooth boundary $\partial\Omega$,
$\lambda,\theta\in\mathbb{R}$, $\mu=1$ and $2^{*}=\frac{2N}{N-2}$ is the critical Sobolev exponent.
For $\theta>0$, they first proved the local compactness of the corresponding energy functional and
then showed that problem \eqref{P3} admits a positive mountain pass solution (which is also a ground
state solution) for all $\lambda\in\mathbb{R}$ when $N\geq4$. When $\theta<0$, whether or not the
corresponding energy functional is still locally compact is unknown. By applying a mountain pass lemma
without compactness condition they proved that problem \eqref{P3} possesses a nontrivial weak solution
when $N=3,4$ and the parameters fulfill certain conditions. Moreover, a nonexistence result was also
obtained when $N\geq3$, $\theta<0$ and
$-\frac{(N-2)\theta}{2}+\frac{(N-2)\theta}{2}\log(\frac{(N-2)\theta}{2})+\lambda-\lambda_1(\Omega)\geq0$.
It is worthy pointing out that due to the lack of compactness, neither the types nor the energy levels
of the solutions obtained in \cite{Dengyinbin} were specified when $\theta<0$.

Later, Hajaiej et al. \cite{Hajaiej,Ha2024} re-considered problem \eqref{P3} with $N\geq4$,
$\mu>0$ and $\theta<0$. By virtue of Ekeland's variational principle and Br\'{e}zis-Lieb's
lemma, they showed the existence of a positive local minimum solution and a positive least energy
solution when $(\lambda,\theta,\mu)\in \Sigma_1\cup\Sigma_2$, where
\begin{align*}
\Sigma_1:&=\left\{(\lambda, \mu, \theta): \lambda\in \left[0, \lambda_1(\Omega)\right), \mu>0, \theta<0,
 \left(\frac{\lambda_{1}(\Omega)-\lambda}{\lambda_{1}(\Omega)}\right)^{\frac{N}{2}}\mu^{-\frac{N-2}{2}}S^{\frac{N}{2}}
 +\frac{N}{2}\theta|\Omega|>0\right\},\\
\Sigma_2:&=\left\{(\lambda, \mu, \theta): \lambda\in \mathbb{R}, \mu>0, \theta<0,
 \mu^{-\frac{N-2}{2}}S^{\frac{N}{2}}+\frac{N}{2}\theta e^{-\frac{\lambda}{\theta}}|\Omega|>0\right\},
\end{align*}
and $\lambda_{1}(\Omega)$ is the first eigenvalue of $-\Delta$ in $\Omega$ with null Dirichlet boundary
condition. Comparing the results in \cite{Hajaiej,Ha2024} with the ones in \cite{Dengyinbin} one sees
that the authors of \cite{Hajaiej,Ha2024} not only weakened part of the conditions for problem \eqref{P3}
with $\theta<0$ to admit weak solutions, but also specified the types and energy levels of the solutions.
Moreover, since the corresponding energy functional possesses a mountain pass structure around the local
minimum solution, the authors of \cite{Hajaiej} conjectured that besides the local minimum solution,
equation \eqref{P1} possesses a mountain pass solution. Actually they proposed the following conjecture:

{\bf Conjecture:} Equation \eqref{P1} possesses a positive mountain pass solution at level $c_{M}>0$.

The aim of this paper is to give a positive answer to this conjecture. Let us explain our strategies
in a detailed way. Since the equation involves a critical term, the corresponding energy functional
lacks global compactness, and what we can expect is only the local compactness. Consequently, to show the
existence of a mountain pass solution, we need to control the mountain pass level from above by a
proper constant to ensure the local compactness. This is the key step of the proof, and is done
by taking full advantage of the local minimum solution and some very delicate estimates on the logarithmic
term and the critical term acting on the truncated Aubin-Talenti bubbles. Set
\begin{align*}
A_1:&=\left\{(\lambda, \mu, \theta): \lambda\in \left[0, \lambda_1(\Omega)\right), \mu>0, \theta<0,
 \frac{(\lambda_{1}(\Omega)-\lambda)^2}{\lambda_{1}^2(\Omega)\mu}S^2+e\theta|\Omega|\geq0\right\},\\
A_2:&=\left\{(\lambda, \mu, \theta): \lambda\in \mathbb{R}, \mu>0, \theta<0,
 \mu^{-1}S^2+\theta e^{1-\frac{\lambda}{\theta}}|\Omega|\geq0\right\}.
\end{align*}
Then the main result of the paper can be stated as follows.

\begin{theorem}\label{th1.1}
Assume that $(\lambda,\mu,\theta)\in A_1\cup A_2$.
Then problem \eqref{P1} possesses a positive mountain pass solution at level $c_{M}>0$.
\end{theorem}

This paper is organized as follows. In Section 2 we introduce some notations and present some preliminary results.
The key energy estimate $c_{M}<c_{K}+\frac{1}{4}\mu^{-1}S^2$ is proved in Section 3, and the proof of Theorem \ref{th1.1}
is completed in Section 4. Finally, in Section 5 we show that the existence of a mountain pass solution also holds
when $N=3$ and $N=5$.

\section{Preliminaries}

In this section, we first introduce some notations and then present several preliminary results.
Throughout this paper, we use $\|\cdot\|_p$ to denote the norm of $L^p(\Omega)$ for $1\leq p\leq\infty$
and denote the norm of $H_0^1(\Omega)$ by $\|\cdot\|:=\|\nabla \cdot\|_2$. The dual space of $H_0^1(\Omega)$
is denoted by $H^{-1}(\Omega)$ and the dual pair between $H^{-1}(\Omega)$ and $H_0^1(\Omega)$ is written as
$\langle\cdot,\cdot\rangle$. For each Banach space $B$, we use $\rightarrow$ and $\rightharpoonup$ to denote
the strong convergence and weak convergence in $B$, respectively. The notation $|\Omega|$ means the Lebesgue measure of
$\Omega$ in $\mathbb{R}^N$. For each function $u$, we use $u^+$ and $u^{-}$ to denote the positive and negative
parts of $u$, respectively, i.e., $u^+=\max\{u,0\}$, $u^{-}=-\max\{-u,0\}$. The symbol $O(t)$ means $|\frac{O(t)}{t}|\leq C$
as $t\rightarrow 0$ and $o_n(1)$ is an infinitesimal as $n\rightarrow\infty$. The capital letter $C$ will appear as
a generic positive constant which may vary from line to line. The positive constant $S$ denotes the best embedding
constant from $H_0^1(\Omega)$ to $L^{2^{*}}(\Omega)$, i.e.,
\begin{align}\label{S}
S=\inf\limits_{u\in H_0^1(\Omega)\backslash\{0\}}\dfrac{\| u\|^2}{\|u\|_{2^{*}}^{2}}.
\end{align}

To find positive solutions to problem \eqref{P1}, we define the associated modified energy functional
\begin{eqnarray*}
J(u)=\frac{1}{2}\int_{\Omega}|\nabla u|^2\mathrm{d}x-\frac{\lambda}{2}\int_{\Omega}|u^+|^2\mathrm{d}x
-\frac{\mu}{4}\int_{\Omega}|u^+|^4\mathrm{d}x
-\frac{\theta}{2}\int_{\Omega}(u^+)^2\left(\log(u^+)^2-1\right)\mathrm{d}x,\ u\in H_0^1(\Omega).
\end{eqnarray*}
Then the Fr\'{e}chet derivative of $J(u)$ can be expressed as
\begin{eqnarray*}
\langle J'(u),\phi \rangle=\int_{\Omega}\nabla u\nabla \phi\mathrm{d}x-\lambda\int_{\Omega}u^+ \phi\mathrm{d}x- \mu\int_{\Omega}(u^+)^3\phi\mathrm{d}x
-\theta\int_{\Omega}u^+ \phi \log (u^+)^{2}\mathrm{d}x, \ u,\phi \in H_{0}^{1}(\Omega).
\end{eqnarray*}
It is easy to see that $J(u)$ is well defined in $H_0^1(\Omega)$, and any nonnegative critical point of $J$
corresponds to a solution to problem \eqref{P1}.


To deal with the logarithmic term $u\log u^2$, we need the following basic inequalities, whose proofs can
be found in [\cite{LHG2024}, Lemma 2.3].
\begin{lemma}\label{logarithmic inequality}
$(1)$ For all $t\in (0,1]$, there holds that
\begin{equation}\label{log-1}
|t\log t| \leq\frac{1}{e}.
\end{equation}

$(2)$ For any $\delta>0$, there holds that
\begin{equation}\label{log-3}
\log t\leq \frac{1}{e\delta}{t^\delta}, \qquad \forall\ t>0.
\end{equation}
\end{lemma}

Assume that $(\lambda, \mu, \theta)\in A_1\cup A_2$.  Since $(A_1\cup A_2)\subset (\Sigma_1\cup\Sigma_2)$,
we know from \cite{Hajaiej} that problem \eqref{P1} has a positive solution $u_0\in C^2(\Omega)\cap
L^{\infty}(\Omega)$ satisfying
\begin{equation}\label{c-rho}
c_\rho:=\inf\limits_{w\in B_\rho}J(w)=J(u_0)<0,
\end{equation}
where $B_{\rho}=\{u\in H_0^1(\Omega): \|u\|\leq\rho\}$ and
\begin{align}\label{equat1}
\rho=\left(\frac{\lambda_1(\Omega)-\lambda}{\lambda_1(\Omega)\mu}\right)^{\frac{1}{2}}S \quad  \text{if}
\ (\lambda, \mu, \theta)\in A_1, \
and \ \rho=\mu^{-\frac{1}{2}}S \quad \text{if} \ (\lambda, \mu, \theta)\in A_2.
\end{align}
This means that $u_0$ is a local minimum of the energy functional $J(u)$ and fulfills
\begin{align}\label{equality3}
\int_{\Omega}\nabla u_0\nabla v\mathrm{d}x-\lambda\int_{\Omega}u_0^+ v\mathrm{d}x
-\mu\int_{\Omega}(u_0^+)^3v\mathrm{d}x
-\theta\int_{\Omega}u_0^+ v \log (u_0^+)^{2}\mathrm{d}x=0, \ \ \forall\ v\in H_0^1(\Omega).
\end{align}
Moreover, by Lemma A.1 in \cite{Hajaiej} we know that there exists $\alpha>0$ such that
\begin{align}\label{ineqt1}
J(u)\geq \alpha \text{\ for all} \ \|u\|=\rho.
\end{align}

Now we prove that the positive local minimum solution $u_0$ to problem \eqref{P1} is also a
positive least energy solution when $(\lambda, \mu, \theta)\in A_1\cup A_2$. For this, set
\begin{equation}\label{c-k}
K=\{u\in H_0^1(\Omega):\ J'(u)=0\},\qquad c_K:=\inf\limits_{u\in K}J(u).
\end{equation}
By Theorem $A.2$ in \cite{Hajaiej} one knows that $c_K<0$ is achieved at a positive least energy
solution $u_k$ to problem \eqref{P1} when $(\lambda,\theta,\mu)\in \Sigma_1\cup\Sigma_2$.

\begin{lemma}\label{gss}
Assume that $(\lambda, \mu, \theta)\in A_1\cup A_2$. Then the positive local minimum solution
$u_0$ is also a positive least energy solution and $c_{\rho}=c_{K}$, where $c_\rho$ is given in
\eqref{c-rho} and $c_K$ is given in \eqref{c-k}.
\end{lemma}

\begin{proof}
Obviously, $c_{K}\leq c_{\rho}$. We claim that $c_{K}\geq c_{\rho}$ when
$(\lambda, \mu, \theta)\in A_1\cup A_2$. The proof is divided into two cases.

\textbf{Case 1: $(\lambda, \mu, \theta)\in A_1$.} Since $\theta<0$, by \eqref{log-1} we have
\begin{align}\label{ineqt2}
&\frac{\theta}{2}\int_{\Omega}|u_k^{+}|^{2}\mathrm{d}x
-\frac{\theta}{4}\int_{\Omega}(u_k^+)^{2}\log (u_k^+)^{2}\mathrm{d}x\nonumber\\
=&-\frac{\theta}{4}\int_{\Omega}(u_k^+)^{2}(-2+\log (u_k^+)^{2})\mathrm{d}x\nonumber\\
\geq&-\frac{\theta}{4}e^2\int_{\{x\in\Omega: e^{-2}(u_k^+)^{2}\leq1\}}e^{-2}(u_k^+)^{2}\log(e^{-2}(u_k^+)^{2})\mathrm{d}x\\
\geq&~\frac{\theta}{4}e|\Omega|.\nonumber
\end{align}
Then according to $\lambda\in [0,\lambda_1(\Omega))$ and \eqref{ineqt2}, we obtain
\begin{align*}
0>c_{K}&=J(u_k)-\frac{1}{4}\langle J'(u_k),u_k\rangle\nonumber\\
&=\frac{1}{4}\| u_k\|^{2}-\frac{\lambda}{4}\int_{\Omega}|u_k^+|^{2}\mathrm{d}x +\frac{\theta}{2}\int_{\Omega}|u_k^{+}|^{2}\mathrm{d}x
-\frac{\theta}{4}\int_{\Omega}(u_k^+)^{2}\log (u_k^+)^{2}\mathrm{d}x\nonumber\\
&\geq\frac{\lambda_1(\Omega)-\lambda}{4\lambda_1(\Omega)}\| u_k\|^{2}+\frac{\theta}{4}e|\Omega|,
\end{align*}
which means that
$$\|u_k\|<\left(\frac{\lambda_1(\Omega)}{\lambda_1(\Omega)-\lambda}|\theta|e|\Omega|\right)^{\frac{1}{2}}.$$
Since $(\lambda, \mu, \theta)\in A_1$, by \eqref{equat1} we have
$$\|u_k\|\leq \rho,$$
i.e., $u_k\in B_{\rho}$ and consequently $c_{\rho}\leq J(u_k)=c_{K}$ for $(\lambda, \mu, \theta)\in A_1$.

\textbf{Case 2: $(\lambda, \mu, \theta)\in A_2$.} Since $\theta<0$, by a similar argument to \eqref{ineqt2}, we have
\begin{align}\label{ineqt3}
-\frac{\lambda}{4}\int_{\Omega}|u_k^+|^{2}\mathrm{d}x +\frac{\theta}{2}\int_{\Omega}|u_k^{+}|^{2}\mathrm{d}x
-\frac{\theta}{4}\int_{\Omega}(u_k^+)^{2}\log (u_k^+)^{2}\mathrm{d}x
\geq\frac{\theta}{4}e^{1-\frac{\lambda}{\theta}}|\Omega|.
\end{align}
Using \eqref{ineqt3}, one obtains
\begin{align}\label{ineq5}
0>c_{K}&=J(u_k)-\frac{1}{4}\langle J'(u_k),u_k\rangle\nonumber\\
&=\frac{1}{4}\| u_k\|^{2}-\frac{\lambda}{4}\int_{\Omega}|u_k^+|^{2}\mathrm{d}x +\frac{\theta}{2}\int_{\Omega}|u_k^{+}|^{2}\mathrm{d}x
-\frac{\theta}{4}\int_{\Omega}(u_k^+)^{2}\log (u_k^+)^{2}\mathrm{d}x\nonumber\\
&\geq\frac{1}{4}\| u_k\|^{2}+\frac{\theta}{4}e^{1-\frac{\lambda}{\theta}}|\Omega|.
\end{align}
From \eqref{ineq5}, we know
\begin{align*}
\|u_k\|<(|\theta|e^{1-\frac{\lambda}{\theta}}|\Omega|)^{\frac{1}{2}},
\end{align*}
which, together with $(\lambda, \mu, \theta)\in A_2$ and \eqref{equat1}, implies that
$$\|u_k\|\leq \rho.$$
Hence $u_k\in B_{\rho}$ and $c_{\rho}\leq J(u_k)=c_{K}$ for $(\lambda, \mu, \theta)\in A_2$.
Therefore, $c_K= c_{\rho}$ and $u_0$  is also a positive least energy solution. This completes the proof.
\end{proof}



Next we shall prove that $J(u)$ satisfies the mountain pass geometry around $u_0$.

\begin{lemma}\label{MPG}
Assume that $(\lambda,\mu, \theta)\in A_1\cup A_2$. Then

$\mathrm{(i)}$ there exists an $r>\|u_0\|$ such that $J(w)>J(u_0)$ for all $w\in H_0^1(\Omega)$ with $\|w\|=r$;

$\mathrm{(ii)}$ there exists a function $v^{*}\in H_0^1(\Omega)$ such that $\|u_0+v^{*}\|>r$ and $J(u_0+v^{*})\leq J(u_0)$.
\end{lemma}

\begin{proof}
Take $r=\rho$. Then the conclusion of $\mathrm{(i)}$ follows directly from \eqref{ineqt1} and the fact that $J(u_0)=c_\rho<0$.

Next, we prove $\mathrm{(ii)}$. For any $v\in H_0^1(\Omega)\backslash\{0\}$, direct computation shows that
$$J(u_0+\beta v)\rightarrow-\infty, \ \text{as} \ \ \beta\rightarrow+\infty.$$
Consequently, there exists $\beta^{*}$ suitably large such that
$$\|u_0+\beta^{*}v\|>r\ \  \text{and} \ \ J(u_0+\beta^{*}v)\leq J(u_0).$$
Set $v^{*}=\beta^{*}v$. The proof of $\mathrm{(ii)}$ is complete.
\end{proof}

Now we define
$$\Gamma:=\left\{\gamma\in C([0,1],H_0^1(\Omega)):\gamma(0)=u_0,J(\gamma(1))\leq J(u_0)\right\},$$
and
$$c_{M}:=\inf_{\gamma\in \Gamma}\max_{t\in[0,1]}J(\gamma(t)).$$
Then we know from Lemma \ref{MPG} and \eqref{ineqt1} that $c_{M}\geq\alpha>0$.


\begin{lemma}\label{PS-condition}
If $(\lambda, \mu, \theta)\in A_1\cup A_2$ and
\begin{align*}
c<c_{K}+\frac{1}{4}\mu^{-1}S^2,
\end{align*}
then $J(u)$ satisfies the $(PS)_c$ condition.
\end{lemma}

\begin{proof}
Assume that $\{u_n\}\subset H_0^1(\Omega)$ is a $(PS)_c$ sequence of $J(u)$, i.e., as $n\rightarrow \infty$,
$$J(u_n)\rightarrow c \ \text{and} \ J'(u_n)\rightarrow0 \ in \ H^{-1}(\Omega).$$
Then  using \eqref{ineqt3}, we have, for n large enough,
\begin{align*}
c+1+o_n(1)\|u_n\|&\geq
J(u_n)-\frac{1}{4}\langle J'(u_n),u_n\rangle\nonumber\\
&=\frac{1}{4}\| u_n\|^{2}-\frac{\lambda}{4}\int_{\Omega}|u_n^+|^{2}\mathrm{d}x +\frac{\theta}{2}\int_{\Omega}|u_n^{+}|^{2}\mathrm{d}x
-\frac{\theta}{4}\int_{\Omega}(u_n^+)^{2}\log (u_n^+)^{2}\mathrm{d}x\nonumber\\
&\geq\frac{1}{4}\| u_n\|^{2}+\frac{\theta}{4}e^{1-\frac{\lambda}{\theta}}|\Omega|,\nonumber
\end{align*}
which implies that $\{u_n\}$ is bounded in $H_0^1(\Omega)$. As a consequence, there exist a subsequence
of $\{u_n\}$ (still denoted by $\{u_n\}$) and a $u\in H_0^1(\Omega)$ such that as $n\rightarrow\infty$,
\begin{align}\label{convergence}
&u_{n}\rightharpoonup u \ in \ H_{0}^1(\Omega),\nonumber\\
&u_{n}\rightarrow u \ in \ L^p(\Omega)\ (1\leq p<4),\\
&|u_{n}|^2u_n\rightharpoonup |u|^{2}u \ in \ L^{\frac{4}{3}}(\Omega),\nonumber\\
&u_{n}\rightarrow u \ \ a.e.\ in \ \Omega.\nonumber
\end{align}

To show that $\{u_n\}$ has a strongly convergent subsequence in $H_0^1(\Omega)$, set $w_n=u_n-u$.
Then $\{w_n\}$ is also bounded in $H_0^1(\Omega)$ and there exists a subsequence of $\{w_{n}\}$ which we still denote by $\{w_{n}\}$ such that
\begin{equation}\label{wn1}
\lim\limits_{n\rightarrow+\infty}\|w_n\|^2=k\geq0.
\end{equation}
Moreover, it follows from the weak convergence $u_{n}\rightharpoonup u$ in $H_{0}^1(\Omega)$ that
\begin{eqnarray}\label{wn2}
\|u_n\|^{2}
=\| w_n\|^2+\| u\|^{2}+o_n(1), \qquad  n\rightarrow\infty.
\end{eqnarray}
By using the Br\'{e}zis-Lieb's lemma \cite{BrezisLieb}, we have
\begin{equation}\label{wn3}
\|u_n^+\|_{4}^{4}=\|w_n^+\|_{4}^{4}+
\|u^+\|_{4}^{4}+o_n(1),\qquad n\rightarrow\infty.
\end{equation}
In addition, since $\{ u_n\}$ is bounded in $H_0^1(\Omega)$, $\{u_n^+\}$ is also bounded in $H_0^1(\Omega)$.
Combining this with \eqref{convergence}, we obtain
\begin{align}\label{convergence2}
&u_{n}^+\rightharpoonup u^+ \ in \ H_{0}^1(\Omega),\ n\rightarrow\infty,\nonumber\\
&u_{n}^+\rightarrow u^+ \ in \ L^p(\Omega)\ (1\leq p<4),\ n\rightarrow\infty,\\
&|u_{n}^+|^2u_n^+\rightharpoonup |u^+|^{2}u^+ \ in \ L^{\frac{4}{3}}(\Omega),\ n\rightarrow\infty,\nonumber\\
&u_{n}^+\rightarrow u^+ \ \ a.e.\ in \ \Omega,\ n\rightarrow\infty.\nonumber
\end{align}
Similar to the proof of $(3.4)$ and $(3.7)$ in \cite{Li2023}, one has
\begin{equation}\label{log1}
\lim_{n\rightarrow\infty}\int_{\Omega}(u_n^+)^{2}\log (u_n^+)^2\mathrm{d}x
=\int_{\Omega}(u^+)^{2}\log (u^+)^2\mathrm{d}x,
\end{equation}
and
\begin{equation}\label{log2}
\lim_{n\rightarrow\infty}\int_{\Omega}u_n^+\phi\log (u_n^+)^2\mathrm{d}x
=\int_{\Omega}u^+\phi\log (u^+)^2\mathrm{d}x,\qquad \forall\ \phi\in H_0^1(\Omega).
\end{equation}
Since $J'(u_n)\rightarrow0$ as $n\rightarrow\infty$, using \eqref{convergence}, \eqref{convergence2} and \eqref{log2}, we have, for any $\phi\in H_0^1(\Omega)$,
\begin{align}\label{equa1}
o_n(1)=\langle J'(u_n),\phi\rangle=\langle J'(u),\phi\rangle+o_n(1), \qquad n\rightarrow\infty,
\end{align}
It follows from \eqref{equa1} that $J'(u)=0$ and
\begin{align*}
0=\langle J'(u),u^{-}\rangle=\int_{\Omega}|\nabla u^{-}|^2\mathrm{d}x,
\end{align*}
which means that $u\geq0$ and $u$ is a nonnegative weak solution to problem \eqref{P1}.
Consequently,
\begin{align}\label{ineqt4}
J(u)\geq c_{K}.
\end{align}
Taking $\phi=u$ in \eqref{equa1}, one obtains
\begin{align}\label{equa2}
\langle J'(u),u\rangle=0.
\end{align}

Applying $J'(u_n)\rightarrow0$ as $n\rightarrow\infty$ again, recalling \eqref{wn2}, \eqref{wn3},
\eqref{convergence2}, \eqref{log1}, \eqref{equa2} and the boundedness of $\{u_n\}$, one gets
\begin{align*}
o_n(1)=&\langle J'(u_n),u_n\rangle\\
=&\int_{\Omega}|\nabla u_n|^2\mathrm{d}x-\lambda\int_{\Omega}|u_n^+|^2\mathrm{d}x
-\theta\int_{\Omega}(u_n^+)^2 \log (u_n^+)^{2}\mathrm{d}x- \mu\int_{\Omega}|u_n^+|^4\mathrm{d}x\\
=&\langle J'(u),u\rangle+\|w_n\|^2-\mu\int_{\Omega}|w_n^+|^4\mathrm{d}x+o_n(1)\\
=&\|w_n\|^2-\mu\int_{\Omega}|w_n^+|^4\mathrm{d}x+o_n(1),
\end{align*}
which implies that
\begin{equation}\label{equa3}
\|w_n\|^2=\mu\int_{\Omega}|w_n^+|^4\mathrm{d}x+o_n(1),\qquad \text{as} \  n\rightarrow\infty.
\end{equation}
Since $J(u_n)\rightarrow c$ as $n\rightarrow\infty$, by \eqref{wn2}, \eqref{wn3}, \eqref{convergence2} and \eqref{log1}, we have
\begin{align}\label{equa4}
&c+o_n(1)=J(u_n)\nonumber\\
=&\frac{1}{2}\int_{\Omega}|\nabla u_n|^2\mathrm{d}x-\frac{\lambda}{2}\int_{\Omega}|u_n^+|^2\mathrm{d}x
-\frac{\theta}{2}\int_{\Omega}(u_n^+)^2\left(\log(u_n^+)^2-1\right)\mathrm{d}x\nonumber-\frac{\mu}{4}\int_{\Omega}|u_n^+|^4\mathrm{d}x\nonumber\\
=&J(u)+\frac{1}{2}\| w_n\|^2-\frac{\mu}{4}\int_{\Omega}|w_n^+|^4\mathrm{d}x+o_n(1),\qquad  n\rightarrow\infty.
\end{align}
Letting $n\rightarrow\infty$ in \eqref{equa4} and recalling \eqref{wn1} and \eqref{equa3}, we get
\begin{align*}
J(u)=c-\frac{1}{4}k.
\end{align*}
By \eqref{S},
\begin{align}\label{inequality1}
\left(\int_{\Omega}|w_n^+|^4\mathrm{d}x\right)^{\frac{1}{2}}\leq\left(\int_{\Omega}|w_n|^4\mathrm{d}x\right)^{\frac{1}{2}}
\leq\frac{1}{S}\|w_n\|^2.
\end{align}
Letting $n\rightarrow\infty$ in \eqref{inequality1} and in view of \eqref{wn1} and \eqref{equa3}, we can deduce that
\begin{align*}
(\frac{1}{\mu})^{\frac{1}{2}}k^{\frac{1}{2}}\leq\frac{1}{S}k.
\end{align*}
If $k>0$, we have $k\geq\mu^{-1}S^2$ since $\mu>0$. Then
\begin{align}\label{equa5}
J(u)=c-\frac{1}{4}k\leq c-\frac{1}{4}\mu^{-1}S^2<c_{K},
\end{align}
which contradicts with \eqref{ineqt4}.
Hence $k=0$, i.e., $u_n\rightarrow u$ in $H_0^1(\Omega)$ as $n\rightarrow\infty$.
The proof is complete.
\end{proof}

\section{Estimates on $c_{M}$}

From Lemma \ref{PS-condition} we know that one needs to control the mountain pass level $c_M$ from above by $c_{K}+\frac{1}{4}\mu^{-1}S^2$
to ensure the compactness of the $(PS)$ sequence $\{u_n\}$ and to obtain the existence of a mountain pass solution. This is the main
task of this section. First we claim that to prove $c_M<c_{K}+\frac{1}{4}\mu^{-1}S^2$, it suffices to show the existence of a function $\widetilde{v}\in H_0^1(\Omega)\setminus\{0\}$
such that
\begin{equation*}
\sup_{\beta\geq0}J(u_0+\beta \widetilde{v})<c_{K}+\frac{1}{4}\mu^{-1}S^2.
\end{equation*}
Indeed, from the proof of $\mathrm{(ii)}$ of Lemma \ref{MPG}, we know that there exists $\beta^{*}$ such that $\|u_0+\beta^{*}\widetilde{v}\|>r$ and
$J(u_0+\beta^{*}\widetilde{v})\leq J(u_0)$. Then we have
\begin{align}\label{ineqt5}
0<c_{M}=\inf_{\gamma\in \Gamma}\max_{t\in [0,1]}J(\gamma(t))
\leq\max_{t\in [0,1]}J(u_0+t\beta^{*}\widetilde{v})\leq\sup_{\beta\geq0}J(u_0+\beta\widetilde{v})<c_{K}+\frac{1}{4}\mu^{-1}S^2.
\end{align}

Without loss of generality, we may assume that $0\in \Omega$ and there exists $\varrho >0$ such that $B_{2\varrho}(0)\subset\Omega$ and
\begin{align}\label{eq1}
m_{\varrho}:=\inf_{x\in B_{\varrho}(0)}u_0(x)\geq\frac{1}{2}u_0(0).
\end{align}
Notice that \eqref{eq1} is guaranteed by the continuity of $u_0$. For $\varepsilon>0$ and $y\in \mathbb{R}^4$, we consider the Aubin-Talenti bubble (\cite{Aubin1976},\cite{Talenti1976}) $U_{\varepsilon,y}(x)\in D^{1,2}(\mathbb{R}^4)$
defined by
$$U_{\varepsilon,y}(x)=\frac{2\sqrt{2}\varepsilon}{\varepsilon^2+|x-y|^2}, \ \ x\in \mathbb{R}^4.$$
It is well known that $U_{\varepsilon,y}$ is a solution to the following equation
$$-\Delta u=u^{3} \ \ in \ \ \mathbb{R}^4,$$
and
$$\int_{\mathbb{R}^4}|\nabla U_{\varepsilon,y}|^2\mathrm{d}x=\int_{\mathbb{R}^4}|U_{\varepsilon,y}|^4\mathrm{d}x=S^2.$$
Here $S$ is the best Sobolev constant of $D^{1,2}(\mathbb{R}^4)\hookrightarrow L^{4}(\mathbb{R}^4)$, characterized as
$$S=\inf\limits_{u\in D^{1,2}(\mathbb{R}^4)\backslash\{0\}}\frac{\int_{\mathbb{R}^4}|\nabla U_{\varepsilon,y}|^2\mathrm{d}x}
{(\int_{\mathbb{R}^4}|U_{\varepsilon,y}|^4\mathrm{d}x)^{\frac{1}{2}}},$$
where $D^{1,2}(\mathbb{R}^4)=\{u\in L^2(\mathbb{R}^4):|\nabla u|\in L^2(\mathbb{R}^4)\}$ with norm
$\|u\|_{D^{1,2}}:=\left(\int_{\mathbb{R}^4}|\nabla u|^2\mathrm{d}x\right)^{\frac{1}{2}}$.
Note that $S$ is the same as the constant defined in \eqref{S} when $N=4$.

Let $\varphi(x)\in C_0^{\infty}(\Omega)$ be a radial function such that $\varphi(x)\equiv1$ for $0\leq|x|\leq\varrho$;
$0\leq\varphi(x)\leq1$ for $\varrho\leq|x|\leq2\varrho$ and $\varphi(x)\equiv0$ for $|x|\geq2\varrho$. Define
\begin{align}
\psi_{\varepsilon}(x)=\varphi(x)u_{\varepsilon}(x),
\end{align}
where $u_{\varepsilon}(x)=U_{\varepsilon,0}(x)$. Then we have the following estimates (see [\cite{Capozzi1985}, Lemma 2.1]
and [\cite{Willem}, Lemma 1.46]).

\begin{lemma}\label{estimate1}
As $\varepsilon\rightarrow 0^+$,
\begin{eqnarray*}
\int_{\Omega}|\nabla \psi_{\varepsilon}|^2\mathrm{d}x=S^2+O(\varepsilon^2),
\end{eqnarray*}
\begin{eqnarray*}
\int_{\Omega}|\psi_{\varepsilon}|\mathrm{d}x=O(\varepsilon),
\end{eqnarray*}
\begin{eqnarray*}\label{es2}
\int_{\Omega}|\psi_{\varepsilon}|^2\mathrm{d}x=d\varepsilon^2|\log\varepsilon|+O(\varepsilon^2),
\end{eqnarray*}
\begin{eqnarray*}
\int_{\Omega}|\psi_{\varepsilon}|^3\mathrm{d}x=O(\varepsilon),
\end{eqnarray*}
and
\begin{eqnarray*}\label{es3}
\int_{\Omega}|\psi_{\varepsilon}|^4\mathrm{d}x=S^2+O(\varepsilon^4),
\end{eqnarray*}
where $d$ is a positive constant.
\end{lemma}

To compare the mountain pass level $c_M$ with $c_{K}+\frac{1}{4}\mu^{-1}S^2$, we need to estimate the logarithmic term.

\begin{lemma}\label{estimate2}
For any $\beta >0$, we have
\begin{align*}
&\int_{\Omega}(u_0+\beta \psi_{\varepsilon})^2\log(u_0+\beta \psi_{\varepsilon})^2\mathrm{d}x
-\int_{\Omega}u_0^2\log u_0^2\mathrm{d}x-\int_{\Omega}2\beta u_0 \psi_{\varepsilon}\log u_0^2\mathrm{d}x\\
\leq&~\frac{8}{e}\beta^2\int_{\Omega}u_0^{\frac{1}{2}}\psi_{\varepsilon}^2\mathrm{d}x
+\frac{8}{e}\beta^{\frac{5}{2}}\int_{\Omega}\psi_{\varepsilon}^{\frac{5}{2}}\mathrm{d}x
+2\beta\int_{\Omega}u_0\psi_{\varepsilon}\mathrm{d}x+6\beta^2\int_{\Omega}\psi_{\varepsilon}^2\mathrm{d}x.
\end{align*}
\end{lemma}

\begin{proof}
Let us begin to make the following decomposition
\begin{align*}
&\int_{\Omega}(u_0+\beta \psi_{\varepsilon})^2\log(u_0+\beta \psi_{\varepsilon})^2\mathrm{d}x
-\int_{\Omega}u_0^2\log u_0^2\mathrm{d}x-\int_{\Omega}2\beta u_0 \psi_{\varepsilon}\log u_0^2\mathrm{d}x\\
=&\int_{\Omega}\left[(u_0^2+\beta^2\psi_{\varepsilon}^2)\log(u_0+\beta \psi_{\varepsilon})^2-u_0^2\log u_0^2\right]\mathrm{d}x\\
&+\int_{\Omega}\left[2\beta u_0\psi_{\varepsilon}\log(u_0+\beta \psi_{\varepsilon})^2-2\beta u_0 \psi_{\varepsilon}\log u_0^2\right]\mathrm{d}x\\
=&:I_1+I_2.
\end{align*}
Recall the following basic inequalitys:
\begin{align}\label{ineq1}
C(a^k+b^k)\leq(a+b)^k\leq a^k+b^k, \ \text{for} \ \ a,b\geq0, \ \text{and} \ k\in(0,1).
\end{align}
And for $a,b\geq0$, $k\geq1$, one has
\begin{align}\label{ineq1-2}
a^k+b^k\leq(a+b)^k\leq C(a^k+b^k).
\end{align}
Then by \eqref{log-3} with $\delta=\frac{1}{4}$ and \eqref{ineq1}, we have
\begin{align}\label{ineq2}
I_1&=\int_{\Omega}\left[(u_0^2+\beta^2\psi_{\varepsilon}^2)\log(u_0+\beta \psi_{\varepsilon})^2-u_0^2\log u_0^2\right]\mathrm{d}x\nonumber\\
&=\int_{\Omega}\int_0^{\beta}\frac{\partial}{\partial\xi}\left[(u_0^2+\xi^2\psi_{\varepsilon}^2)
\log(u_0+\xi\psi_{\varepsilon})^2\right]\mathrm{d}\xi\mathrm{d}x\nonumber\\
&=\int_{\Omega}\int_0^{\beta}2\xi\psi_{\varepsilon}^2\log(u_0+\xi\psi_{\varepsilon})^2
+\frac{2\psi_{\varepsilon}(u_0^2+\xi^2\psi_{\varepsilon}^2)}{u_0+\xi\psi_{\varepsilon}}\mathrm{d}\xi\mathrm{d}x\nonumber\\
&\leq\int_{\Omega}\int_0^{\beta}\frac{8}{e}\xi\psi_{\varepsilon}^2(u_0+\xi\psi_{\varepsilon})^{\frac{1}{2}}
+2u_0\psi_{\varepsilon}+2\xi\psi_{\varepsilon}^2\mathrm{d}\xi\mathrm{d}x\nonumber\\
&\leq\int_{\Omega}\int_0^{\beta}\frac{8}{e}\xi\psi_{\varepsilon}^2(u_0^{\frac{1}{2}}+\xi^{\frac{1}{2}}\psi_{\varepsilon}^{\frac{1}{2}})
+2u_0\psi_{\varepsilon}+2\xi\psi_{\varepsilon}^2\mathrm{d}\xi\mathrm{d}x\nonumber\\
&\leq\int_{\Omega}\frac{8}{e}\beta^2 u_0^{\frac{1}{2}}\psi_{\varepsilon}^2
+\frac{8}{e}\beta^{\frac{5}{2}}\psi_{\varepsilon}^{\frac{5}{2}}
+2\beta u_0\psi_{\varepsilon}+2\beta^2\psi_{\varepsilon}^2\mathrm{d}x.
\end{align}
As for $I_2$, we can estimate it as follows
\begin{align}\label{ineq3}
I_2&=\int_{\Omega}2\beta u_0\psi_{\varepsilon}\log(u_0+\beta \psi_{\varepsilon})^2-2\beta u_0 \psi_{\varepsilon}\log u_0^2\mathrm{d}x\nonumber\\
&=\int_{\Omega}\int_0^1\frac{\partial}{\partial\xi}\left[2\beta u_0\psi_{\varepsilon}\log(u_0
+\xi\beta \psi_{\varepsilon})^2\right]\mathrm{d}\xi\mathrm{d}x\nonumber\\
&=\int_{\Omega}\int_0^1 2\beta u_0\psi_{\varepsilon}\cdot\frac{2\beta \psi_{\varepsilon}}{u_0+\xi\beta \psi_{\varepsilon}}\mathrm{d}\xi\mathrm{d}x\nonumber\\
&\leq\int_{\Omega}4\beta^2\psi_{\varepsilon}^2\mathrm{d}x.
\end{align}
From \eqref{ineq2} and \eqref{ineq3}, we arrive at
\begin{align*}
&\int_{\Omega}(u_0+\beta \psi_{\varepsilon})^2\log(u_0+\beta \psi_{\varepsilon})^2\mathrm{d}x
-\int_{\Omega}u_0^2\log u_0^2\mathrm{d}x-\int_{\Omega}2\beta u_0 \psi_{\varepsilon}\log u_0^2\mathrm{d}x\\
\leq&\frac{8}{e}\beta^2\int_{\Omega} u_0^{\frac{1}{2}}\psi_{\varepsilon}^2\mathrm{d}x
+\frac{8}{e}\beta^{\frac{5}{2}}\int_{\Omega}\psi_{\varepsilon}^{\frac{5}{2}}\mathrm{d}x
+2\beta\int_{\Omega} u_0\psi_{\varepsilon}\mathrm{d}x
+6\beta^2\int_{\Omega}\psi_{\varepsilon}^2\mathrm{d}x.
\end{align*}
This completes the proof.
\end{proof}

\begin{lemma}\label{estimate3}
As $\varepsilon\rightarrow 0^+$, we have
\begin{align}\label{est1}
\int_{\Omega}|\psi_{\varepsilon}|^{\frac{5}{2}}\mathrm{d}x=O(\varepsilon^{\frac{3}{2}}),
\end{align}
and
\begin{align}\label{est2}
\int_{\Omega}u_0|\psi_{\varepsilon}|^{3}\mathrm{d}x\geq 8\sqrt{2}u_0(0)\varepsilon\int_{\mathbb{R}^4}\frac{1}{(1+|y|^2)^3}\mathrm{d}y
+ O(\varepsilon^{3}).
\end{align}
\end{lemma}

\begin{proof}
From the definition of $\psi_{\varepsilon}$, we have
\begin{align}\label{1}
\int_{\Omega}|\psi_{\varepsilon}|^{\frac{5}{2}}\mathrm{d}x
=&2^{\frac{15}{4}}\int_{B_{2\varrho}(0)}{\varphi^{\frac{5}{2}}(x)}\frac{\varepsilon^{\frac{5}{2}}}{(\varepsilon^2
+|x|^2)^{\frac{5}{2}}}\mathrm{d}x\nonumber\\
=&2^{\frac{15}{4}}\int_{B_{\varrho}(0)}\frac{\varepsilon^{\frac{5}{2}}}{(\varepsilon^2+|x|^2)^{\frac{5}{2}}}\mathrm{d}x
+2^{\frac{15}{4}}\int_{B_{2\varrho}(0)\backslash B_{\varrho}(0)}
{\varphi^{\frac{5}{2}}(x)}\frac{\varepsilon^{\frac{5}{2}}}{(\varepsilon^2+|x|^2)^{\frac{5}{2}}}\mathrm{d}x\nonumber\\
=&:J_1+J_2.
\end{align}
Direct calculation guarantees that
\begin{align}\label{2}
J_1=&2^{\frac{15}{4}}\int_{B_{\varrho}(0)}\frac{\varepsilon^{\frac{5}{2}}}{(\varepsilon^2+|x|^2)^{\frac{5}{2}}}\mathrm{d}x\nonumber\\
=&2^{\frac{15}{4}}\varepsilon^{\frac{3}{2}}\int_{B_{\varrho/\varepsilon}(0)}\frac{1}{(1+|y|^2)^{\frac{5}{2}}}\mathrm{d}y\nonumber\\
=&2^{\frac{15}{4}}\varepsilon^{\frac{3}{2}}\int_{\mathbb{R}^4}\frac{1}{(1+|y|^2)^{\frac{5}{2}}}\mathrm{d}y
-2^{\frac{15}{4}}\varepsilon^{\frac{3}{2}}\int_{B_{\varrho/\varepsilon}^{c}(0)}\frac{1}{(1+|y|^2)^{\frac{5}{2}}}\mathrm{d}y\nonumber\\
=&:J_{11}-J_{12}.
\end{align}
As $\varepsilon\rightarrow 0^+$,
\begin{align}\label{3}
J_{12}=&2^{\frac{15}{4}}\varepsilon^{\frac{3}{2}}\int_{B_{\varrho/\varepsilon}^{c}(0)}\frac{1}{(1+|y|^2)^{\frac{5}{2}}}\mathrm{d}y\nonumber\\
=&2^{\frac{15}{4}}\omega_4\varepsilon^{\frac{3}{2}}\int_{\frac{\varrho}{\varepsilon}}^{+\infty}\frac{r^3}{(1+r^2)^{\frac{5}{2}}}\mathrm{d}r\nonumber\\
=&2^{\frac{15}{4}}\omega_4\varepsilon^{\frac{3}{2}}\left(-\frac{1}{\sqrt{1+r^2}}\bigg|_{\frac{\varrho}{\varepsilon}}^{+\infty}
+\frac{1}{3(1+r^2)^{\frac{3}{2}}}\bigg|_{\frac{\varrho}{\varepsilon}}^{+\infty}\right)\nonumber\\
=&2^{\frac{15}{4}}\omega_4\varepsilon^{\frac{3}{2}}\left(\frac{\varepsilon}{(\varepsilon^2+\varrho^2)^{\frac{1}{2}}}
-\frac{\varepsilon^{3}}{3(\varepsilon^2+\varrho^2)^{\frac{3}{2}}}\right)\nonumber\\
=&O(\varepsilon^{\frac{5}{2}}).
\end{align}
Now we estimate $J_2$. As $\varepsilon\rightarrow 0^+$, one has
\begin{align}\label{4}
J_2=&2^{\frac{15}{4}}\int_{B_{2\varrho}(0)\backslash B_{\varrho}(0)}
{\varphi^{\frac{5}{2}}(x)}\frac{\varepsilon^{\frac{5}{2}}}{(\varepsilon^2+|x|^2)^{\frac{5}{2}}}\mathrm{d}x\nonumber\\
\leq&2^{\frac{15}{4}}\int_{B_{2\varrho}(0)\backslash B_{\varrho}(0)}\frac{\varepsilon^{\frac{5}{2}}}{(\varepsilon^2+|x|^2)^{\frac{5}{2}}}\mathrm{d}x\nonumber\\
\leq&2^{\frac{15}{4}}\varepsilon^{\frac{5}{2}}\int_{B_{2\varrho}(0)\backslash B_{\varrho}(0)}\frac{1}{|x|^5}\mathrm{d}x\nonumber\\
=&O(\varepsilon^{\frac{5}{2}}).
\end{align}
Then \eqref{est1} follows from \eqref{1}, \eqref{2}, \eqref{3} and \eqref{4}.

Next we prove \eqref{est2}. Notice that $u_0(x)>0$ in $\Omega$. Then by the definition of $\psi_{\varepsilon}$ and \eqref{eq1},
we have
\begin{align}\label{ineq4}
\int_{\Omega}u_0|\psi_{\varepsilon}|^{3}\mathrm{d}x
=&16\sqrt{2}\int_{B_{2\varrho}(0)}u_0(x)\varphi^3(x)\frac{\varepsilon^3}{(\varepsilon^2+|x|^2)^3}\mathrm{d}x\nonumber\\
=&16\sqrt{2}\left(\int_{B_{\varrho}(0)}u_0(x)\frac{\varepsilon^3}{(\varepsilon^2+|x|^2)^3}\mathrm{d}x
+\int_{B_{2\varrho}(0)\backslash B_{\varrho}(0)}u_0(x)\varphi^3(x)\frac{\varepsilon^3}{(\varepsilon^2+|x|^2)^3}\mathrm{d}x\right)\nonumber\\
\geq&16\sqrt{2}\int_{B_{\varrho}(0)}u_0(x)\frac{\varepsilon^3}{(\varepsilon^2+|x|^2)^3}\mathrm{d}x\nonumber\\
\geq&8\sqrt{2}u_0(0)\int_{B_{\varrho}(0)}\frac{\varepsilon^3}{(\varepsilon^2+|x|^2)^3}\mathrm{d}x\nonumber\\
=&8\sqrt{2}u_0(0)\varepsilon\int_{B_{\varrho/\varepsilon}(0)}\frac{1}{(1+|y|^2)^3}\mathrm{d}y\nonumber\\
=&8\sqrt{2}u_0(0)\varepsilon\int_{\mathbb{R}^4}\frac{1}{(1+|y|^2)^3}\mathrm{d}y
-8\sqrt{2}u_0(0)\varepsilon\int_{B_{\varrho/\varepsilon}^{c}(0)}\frac{1}{(1+|y|^2)^3}\mathrm{d}y\nonumber\\
=&:K_1-K_2.
\end{align}
The first term $K_1$ is what we need and the second term $K_2$ can be estimated as follows:
\begin{align}\label{eq3}
K_2=&8\sqrt{2}u_0(0)\varepsilon\int_{B_{\varrho/\varepsilon}^{c}(0)}\frac{1}{(1+|y|^2)^3}\mathrm{d}y\nonumber\\
=&8\sqrt{2}u_0(0)\omega_4\varepsilon\int_{\frac{\varrho}{\varepsilon}}^{+\infty}\frac{r^3}{(1+r^2)^3}\mathrm{d}r\nonumber\\
=&8\sqrt{2}u_0(0)\omega_4\varepsilon\left(-\frac{1}{2(1+r^2)}\bigg|_{\frac{\varrho}{\varepsilon}}^{+\infty}
+\frac{1}{4(1+r^2)^2}\bigg|_{\frac{\varrho}{\varepsilon}}^{+\infty}\right)\nonumber\\
=&8\sqrt{2}u_0(0)\omega_4\varepsilon\left(\frac{\varepsilon^2}{2(\varepsilon^2+\varrho^2)}
-\frac{\varepsilon^4}{4(\varepsilon^2+\varrho^2)^2}\right)\nonumber\\
=&O(\varepsilon^3),\qquad\text{as}\ \varepsilon\rightarrow 0^+.
\end{align}
Hence, \eqref{est2} follows from \eqref{ineq4} and \eqref{eq3}. The proof is complete.
\end{proof}

With Lemmas \ref{estimate1}-\ref{estimate3} at hand, we can now control the mountain pass level around $u_0$ from above
by the desired constant. More precisely, we have the following

\begin{lemma}\label{estimate4}
Assume that $(\lambda,\mu,\theta)\in A_1\cup A_2$. Then for suitably small $\varepsilon>0$ we have
\begin{align}\label{key}
\sup_{\beta\geq0}J(u_0+\beta \psi_{\varepsilon})<c_{K}+\frac{1}{4}\mu^{-1}S^2.
\end{align}
\end{lemma}

\begin{proof}
By Lemma \ref{MPG} and the fact that $J(u_0+\beta \psi_{\varepsilon})\rightarrow-\infty$ as $\beta\rightarrow+\infty$,
we know that there exists a constant $\beta_{\varepsilon}>0$ such that
\begin{align}\label{ineqt6}
\sup_{\beta\geq0}J(u_0+\beta \psi_{\varepsilon})=J(u_0+\beta_{\varepsilon} \psi_{\varepsilon}).
\end{align}

First we claim that there exist constants $C_0, C^0>0$ such that
\begin{align}\label{ineqt7}
C_0 \leq\beta_{\varepsilon}\leq C^0
\end{align}
uniformly for suitably small $\varepsilon$. Indeed, by \eqref{ineqt6}, we have
$$J_{\beta}'(u_0+\beta_{\varepsilon} \psi_{\varepsilon})=0, \ \ \text{and} \ \ \ J_{\beta}''(u_0+\beta_{\varepsilon} \psi_{\varepsilon})\leq0,$$
which then imply
\begin{align}\label{eq5}
&\int_{\Omega}\left(\nabla u_0\nabla\psi_{\varepsilon}-\lambda u_0\psi_{\varepsilon}\right)\mathrm{d}x
+\beta_{\varepsilon}\int_{\Omega}\left(|\nabla\psi_{\varepsilon}|^2-\lambda |\psi_{\varepsilon}|^2\right)\mathrm{d}x\nonumber\\
=&\mu\int_{\Omega}(u_0+\beta_{\varepsilon}\psi_{\varepsilon})^3\psi_{\varepsilon}\mathrm{d}x
+\theta\int_{\Omega}(u_0+\beta_{\varepsilon}\psi_{\varepsilon})\psi_{\varepsilon}\log(u_0+\beta_{\varepsilon}\psi_{\varepsilon})^2\mathrm{d}x,
\end{align}
and
\begin{align}\label{eq6}
\int_{\Omega}|\nabla\psi_{\varepsilon}|^2-(\lambda+2\theta)|\psi_{\varepsilon}|^2\mathrm{d}x
-3\mu\int_{\Omega}|\psi_{\varepsilon}|^2(u_0+\beta_{\varepsilon}\psi_{\varepsilon})^2\mathrm{d}x
-\theta\int_{\Omega}|\psi_{\varepsilon}|^2\log(u_0+\beta_{\varepsilon}\psi_{\varepsilon})^2\mathrm{d}x\leq0.
\end{align}
Since $\mu>0$ and $\theta<0$, using \eqref{log-3} with $\delta=\frac{1}{2}$ and \eqref{ineq1-2}, we deduce from \eqref{eq5} that
\begin{align*}
&\int_{\Omega}\left(\nabla u_0\nabla\psi_{\varepsilon}-\lambda u_0\psi_{\varepsilon}\right)\mathrm{d}x
+\beta_{\varepsilon}\int_{\Omega}\left(|\nabla\psi_{\varepsilon}|^2-\lambda |\psi_{\varepsilon}|^2\right)\mathrm{d}x\nonumber\\
\geq&\mu\int_{\Omega}u_0^3\psi_{\varepsilon}\mathrm{d}x+\mu\beta_{\varepsilon}^3\int_{\Omega}|\psi_{\varepsilon}|^4\mathrm{d}x
+\frac{2\theta}{e}\int_{\Omega}(u_0+\beta_{\varepsilon}\psi_{\varepsilon})^2\psi_{\varepsilon}\mathrm{d}x\nonumber\\
\geq&\mu\int_{\Omega}u_0^3\psi_{\varepsilon}\mathrm{d}x+\mu\beta_{\varepsilon}^3\int_{\Omega}|\psi_{\varepsilon}|^4\mathrm{d}x
+\frac{4\theta}{e}\int_{\Omega}u_0^2\psi_{\varepsilon}\mathrm{d}x+\frac{4\theta}{e}\beta_{\varepsilon}^2\int_{\Omega}|\psi_{\varepsilon}|^3\mathrm{d}x,
\end{align*}
which, together with \eqref{equality3}, shows that
\begin{align*}
&\beta_{\varepsilon}\int_{\Omega}\left(|\nabla\psi_{\varepsilon}|^2-\lambda |\psi_{\varepsilon}|^2\right)\mathrm{d}x
+\theta\int_{\Omega}u_0\psi_{\varepsilon}\log u_0^2\mathrm{d}x-\frac{4\theta}{e}\int_{\Omega}u_0^2\psi_{\varepsilon}\mathrm{d}x\\
\geq&\mu\beta_{\varepsilon}^3\int_{\Omega}|\psi_{\varepsilon}|^4\mathrm{d}x
+\frac{4\theta}{e}\beta_{\varepsilon}^2\int_{\Omega}|\psi_{\varepsilon}|^3\mathrm{d}x.
\end{align*}
Recalling that $u_0\in L^{\infty}(\Omega)$, $\mu>0$ and $\theta<0$, combining the above inequality with Lemma \ref{estimate1},
one has, for suitably small $\varepsilon$,
\begin{align*}
C_1+2S^2\beta_{\varepsilon}\geq \frac{\mu}{2}S^2\beta_{\varepsilon}^3-C_2\beta_{\varepsilon}^2,
\end{align*}
where $C_1$ and $C_2$ are positive constants. Consequently, there exists $C^0>0$ such that $\beta_{\varepsilon}\leq C^0$ uniformly for suitably small $\varepsilon$.

Similarly, using \eqref{ineq1-2}, we deduce from \eqref{eq6} that
\begin{align*}
\int_{\Omega}|\nabla\psi_{\varepsilon}|^2-(\lambda+2\theta)|\psi_{\varepsilon}|^2\mathrm{d}x
-6\mu\int_{\Omega}u_0^2|\psi_{\varepsilon}|^2\mathrm{d}x
-\theta\int_{\Omega}|\psi_{\varepsilon}|^2\log {u_0^2}\mathrm{d}x
\leq 6\mu\beta_{\varepsilon}^2\int_{\Omega}|\psi_{\varepsilon}|^4\mathrm{d}x.
\end{align*}
Applying $u_0\in L^{\infty}(\Omega)$ and Lemma \ref{estimate1} again, one obtains, for suitably small $\varepsilon$,
\begin{align*}
\frac{1}{2}S^2 \leq12\mu S^2\beta_{\varepsilon}^2,
\end{align*}
which implies that there exists $C_0>0$ such that $\beta_{\varepsilon}\geq C_0$ uniformly for suitably small $\varepsilon$.

On the basis of \eqref{ineqt7}, we can estimate $J(u_0+\beta_{\varepsilon} \psi_{\varepsilon})$ as follows
\begin{align*}
J(u_0+\beta_{\varepsilon} \psi_{\varepsilon})=&\frac{1}{2}\int_{\Omega}\left(|\nabla u_0|^2-\lambda|u_0|^2+\theta|u_0|^2\right)\mathrm{d}x
+\frac{\beta_{\varepsilon}^2}{2}\int_{\Omega}\left(|\nabla \psi_{\varepsilon}|^2-\lambda|\psi_{\varepsilon}|^2
+\theta|\psi_{\varepsilon}|^2\right)\mathrm{d}x\nonumber\\
&+\beta_{\varepsilon}\int_{\Omega}\nabla u_0\nabla{\psi_{\varepsilon}}\mathrm{d}x
-\lambda\beta_{\varepsilon}\int_{\Omega}u_0\psi_{\varepsilon}\mathrm{d}x
+\theta\beta_{\varepsilon}\int_{\Omega}u_0\psi_{\varepsilon}\mathrm{d}x
-\frac{\mu}{4}\int_{\Omega}|u_0|^4\mathrm{d}x\nonumber\\
&-\mu\beta_{\varepsilon}\int_{\Omega}u_0^3\psi_{\varepsilon}\mathrm{d}x
-\frac{3\mu}{2}\beta_{\varepsilon}^2\int_{\Omega}u_0^2|\psi_{\varepsilon}|^2\mathrm{d}x
-\mu\beta_{\varepsilon}^3\int_{\Omega}u_0|\psi_{\varepsilon}|^3\mathrm{d}x
-\frac{\mu\beta_{\varepsilon}^4}{4}\int_{\Omega}|\psi_{\varepsilon}|^4\mathrm{d}x\nonumber\\
&-\frac{\theta}{2}\int_{\Omega}(u_0+\beta_{\varepsilon} \psi_{\varepsilon})^2\log{(u_0+\beta_{\varepsilon} \psi_{\varepsilon})^2}\mathrm{d}x.
\end{align*}
Then by Lemma \ref{estimate2} and \eqref{equality3}, we further obtain
\begin{align}\label{inequt8}
J(u_0+\beta \psi_{\varepsilon})\leq&\frac{1}{2}\int_{\Omega}\left(|\nabla u_0|^2-\lambda|u_0|^2+\theta|u_0|^2\right)\mathrm{d}x
+\frac{\beta_{\varepsilon}^2}{2}\int_{\Omega}\left(|\nabla \psi_{\varepsilon}|^2
-\lambda|\psi_{\varepsilon}|^2+\theta|\psi_{\varepsilon}|^2\right)\mathrm{d}x\nonumber\\
&+\beta_{\varepsilon}\int_{\Omega}\nabla u_0\nabla{\psi_{\varepsilon}}\mathrm{d}x
-\lambda\beta_{\varepsilon}\int_{\Omega}u_0\psi_{\varepsilon}\mathrm{d}x
-\frac{\mu}{4}\int_{\Omega}|u_0|^4\mathrm{d}x
-\mu\beta_{\varepsilon}\int_{\Omega}u_0^3\psi_{\varepsilon}\mathrm{d}x\nonumber\\
&-\frac{3\mu}{2}\beta_{\varepsilon}^2\int_{\Omega}u_0^2|\psi_{\varepsilon}|^2\mathrm{d}x
-\mu\beta_{\varepsilon}^3\int_{\Omega}u_0|\psi_{\varepsilon}|^3\mathrm{d}x
-\frac{\mu\beta_{\varepsilon}^4}{4}\int_{\Omega}|\psi_{\varepsilon}|^4\mathrm{d}x\nonumber\\
&-\frac{\theta}{2}\int_{\Omega}u_0^2\log u_0^2\mathrm{d}x
-\theta\beta_{\varepsilon}\int_{\Omega}u_0 \psi_{\varepsilon}\log u_0^2\mathrm{d}x
-\frac{4\theta}{e}\beta_{\varepsilon}^2\int_{\Omega}u_0^{\frac{1}{2}}|\psi_{\varepsilon}|^2\mathrm{d}x\nonumber\\
&-\frac{4\theta}{e}\beta_{\varepsilon}^{\frac{5}{2}}\int_{\Omega}|\psi_{\varepsilon}|^{\frac{5}{2}}\mathrm{d}x
-3\theta\beta_{\varepsilon}^2\int_{\Omega}|\psi_{\varepsilon}|^2\mathrm{d}x\nonumber\\
=&J(u_0)+\frac{\beta_{\varepsilon}^2}{2}\int_{\Omega}|\nabla \psi_{\varepsilon}|^2\mathrm{d}x
-\frac{\mu\beta_{\varepsilon}^4}{4}\int_{\Omega}|\psi_{\varepsilon}|^4\mathrm{d}x
-\frac{\lambda+5\theta}{2}\beta_{\varepsilon}^2\int_{\Omega}|\psi_{\varepsilon}|^2\mathrm{d}x\nonumber\\
&-\frac{3\mu}{2}\beta_{\varepsilon}^2\int_{\Omega}u_0^2|\psi_{\varepsilon}|^2\mathrm{d}x
-\mu\beta_{\varepsilon}^3\int_{\Omega}u_0|\psi_{\varepsilon}|^3\mathrm{d}x
-\frac{4\theta}{e}\beta_{\varepsilon}^2\int_{\Omega}u_0^{\frac{1}{2}}|\psi_{\varepsilon}|^2\mathrm{d}x\nonumber\\
&-\frac{4\theta}{e}\beta_{\varepsilon}^{\frac{5}{2}}\int_{\Omega}|\psi_{\varepsilon}|^{\frac{5}{2}}\mathrm{d}x.
\end{align}
Set $g(t)=\frac{t^2}{2}-\frac{\mu t^4}{4}$, $t>0$. Then $g'(t)>0$ for $0<t<\mu^{-\frac{1}{2}}$,
$g'(t)<0$ for $t>\mu^{-\frac{1}{2}}$ and $g(t)$ attains its maximum at $t=\mu^{-\frac{1}{2}}$ with
\begin{align}\label{66}
\max_{t>0}g(t)=g(\mu^{-\frac{1}{2}})=\frac{1}{4}\mu^{-1}.
\end{align}

By the boundedness of $u_0$ and $\beta_{\varepsilon}$, and applying Lemmas \ref{gss}, \ref{estimate1} and \ref{estimate3},
we obtain from \eqref{inequt8} and \eqref{66} that, as $\varepsilon\rightarrow0^+$,
\begin{align*}
J(u_0+\beta_{\varepsilon} \psi_{\varepsilon})\leq&J(u_0)+\frac{\beta_{\varepsilon}^2}{2}S^2
-\frac{\mu\beta_{\varepsilon}^4}{4}S^2+O(\varepsilon^2)+O(\varepsilon^2|\log \varepsilon|)+O(\varepsilon^{\frac{3}{2}})\\
&-8\sqrt{2}\mu{C_0}^3u_0(0)\varepsilon\int_{\mathbb{R}^N}\frac{1}{(1+|y|^2)^3}\mathrm{d}y+O(\varepsilon^3)\\
<&c_{K}+\frac{1}{4}\mu^{-1}S^2.
\end{align*}
Therefore, \eqref{key} holds with suitably small $\varepsilon>0$ and the proof is complete.
\end{proof}

\section{Proof of Theorem 1.1}

{\bf Proof of Theorem \ref{th1.1}.}
Assume that $(\lambda, \mu, \theta)\in A_1\cup A_2$. From Lemma \ref{MPG} and the Mountain Pass Theorem (Theorem 1.17 in\cite{Willem}), there exists
a sequence $\{u_n\}\subset H_0^1(\Omega)$ such that, as $n\rightarrow\infty$,
$$J(u_n)\rightarrow c_{M} \ \ \ and \ \ \ J'(u_n)\rightarrow 0 \ \  in \ \ H^{-1}(\Omega).$$
This, together with Lemmas \ref{PS-condition}, \ref{estimate4} and \eqref{ineqt5}, implies that
there exists $u\in H_0^1(\Omega)$ such that $u_n\rightarrow u$ in $H_0^1(\Omega)$ and
$$J(u)=c_{M}, \ \ \ \ \ \ \ \  J'(u)=0.$$
Consequently, $\langle J'(u), u^{-}\rangle=\int_{\Omega}|\nabla u^{-}|^2\mathrm{d}x=0$,
which means that $u\geq0$ in $\Omega$. Therefore, $u$ is a nonnegative mountain pass solution to problem \eqref{P1}.
By Morse's iteration, the solution $u$ belong to $L^{\infty}(\Omega)$. Then the H\"{o}lder estimate implies that $u\in C^{0,\gamma}(\Omega)$
for any $0<\gamma<1$. Define $h:[0,+\infty)\rightarrow \mathbb{R}$ by
\begin{align*}
h(s)=
\begin{cases}
\frac{3|\theta|}{2}|s\log s^2|,  &s>0,\\
0, &s=0.
\end{cases}
\end{align*}
Then we follow the arguments in \cite{Dengyinbin, Vazquez1984} and get that $u\in C^2(\Omega)$ and $u>0$ in $\Omega$. This completes the proof.

\section{Appendix}

In this section, we shall show that the conjecture proposed in \cite{Hajaiej} is also valid when $N=3,5$.
Since most of the proof is similar to the case $N=4$, we only sketch the outline here.
To show this, consider the following problem
\begin{eqnarray}\label{P2}
\begin{cases}
-\Delta u= \lambda u+\mu|u|^{2^*-2}u+\theta u\log u^2, &x\in\Omega,\\
u=0, &x\in\partial\Omega,
\end{cases}
\end{eqnarray}
where $\Omega\subset \mathbb{R}^N(N=3,5)$ is a bounded domain with smooth boundary
$\partial\Omega$, $\lambda\in \mathbb{R}$, $\mu>0$, $\theta<0$ and $2^{*}=\frac{2N}{N-2}$ is the Sobolev critical exponent.

The modified energy functional associated with problem \eqref{P2} is defined by
\begin{eqnarray*}
\widetilde{J}(u)=\frac{1}{2}\int_{\Omega}|\nabla u|^2\mathrm{d}x-\frac{\lambda}{2}\int_{\Omega}|u^+|^2\mathrm{d}x-\frac{\mu}{2^*}\int_{\Omega}|u^+|^{2^*}\mathrm{d}x
-\frac{\theta}{2}\int_{\Omega}(u^+)^2\left(\log(u^+)^2-1\right)\mathrm{d}x.
\end{eqnarray*}

Set
\begin{align*}
A_3:&=\left\{(\lambda, \mu, \theta): \lambda\in \left[0, \lambda_1(\Omega)\right), \mu>0, \theta<0,
 \left(\frac{\lambda_{1}(\Omega)-\lambda}{\lambda_{1}(\Omega)}\right)^{\frac{N}{2}}\mu^{-\frac{N-2}{2}}S^{\frac{N}{2}}
 +\theta e^{\frac{N-2}{2}}|\Omega|\geq0\right\},\\
A_4:&=\left\{(\lambda, \mu, \theta): \lambda\in \mathbb{R}, \mu>0, \theta<0,
 \mu^{-\frac{N-2}{2}}S^{\frac{N}{2}}+\theta e^{\frac{N-2}{2}-\frac{\lambda}{\theta}}|\Omega|\geq0\right\}.
\end{align*}

\begin{theorem}\label{th5.1}
Assume that $(\lambda,\mu,\theta)\in A_3\cup A_4$.
Then problem \eqref{P2} with $N=3,5$ possesses a positive mountain pass solution at level $\widetilde{c}_{M}>0$.
\end{theorem}

\begin{remark}\label{555}
In Theorems \ref{th1.1} and \ref{5.1}, we obtain the existence of a positive mountain pass solution to
problem \eqref{P2} with $3\leq N\leq 5$. However, the method can not be used to deal with the case $N\geq 6$,
since we can not control the mountain pass level from above by a proper constant to ensure the local compactness of the functional.
Maybe some new methods need to be applied to consider the case $N\geq 6$.
\end{remark}

The proof of Theorem \ref{th5.1} is completed by a series of lemmas.

\begin{lemma}\label{5.1}
Assume that $N=3,5$ and $(\lambda,\mu,\theta)\in A_3\cup A_4$.
Then problem \eqref{P2} possesses a positive local minimum solution 
$\widetilde{u}_0\in C^2(\Omega)\cap L^{\infty}(\Omega)$ such that $\|\widetilde{u}_0\|\leq\widetilde{\rho}$
and $\widetilde{J}(\widetilde{u}_0)=c_{\widetilde{\rho}}<0$ and a positive least energy solution 
$\widetilde{u}_k\in C^2(\Omega)\cap L^{\infty}(\Omega)$ such that
$\widetilde{J}(\widetilde{u}_k)=\widetilde{c}_K<0$,
where
\begin{align}\label{eq5-1}
\widetilde{\rho}=\left(\frac{\lambda_1(\Omega)-\lambda}{\lambda_1(\Omega)\mu}S^{\frac{2^*}{2}}\right)^{\frac{1}{2^*-2}} \
\text{if} \ (\lambda, \mu, \theta)\in A_3, \
\text{and} \ \widetilde{\rho}=\left(\mu^{-1}S^{\frac{2^*}{2}}\right)^{\frac{1}{2^*-2}} \ \text{if} \ (\lambda, \mu, \theta)\in A_4.
\end{align}
\end{lemma}
\begin{proof}
The proof follows directly from the strategies in [\cite{Hajaiej}, Theorems A.1 and A.2], so we omit the details here.
\end{proof}

Similar to Lemmas \ref{gss}, \ref{MPG}, \ref{PS-condition} we have the following lemmas.

\begin{lemma}\label{5.2}
Assume that $N=3,5$ and $(\lambda,\mu,\theta)\in A_3\cup A_4$. Then $\widetilde{u}_0$ is also a positive least energy solution with $\widetilde{J}(\widetilde{u}_0)=\widetilde{c}_{K}$.
\end{lemma}

\begin{lemma}\label{5.3}
Assume that $N=3,5$ and $(\lambda,\mu,\theta)\in A_3\cup A_4$. Then we have

$\mathrm{(i)}$ there exists an $\widetilde{r}>\|\widetilde{u}_0\|$ such that $\widetilde{J}(w)>\widetilde{J}(\widetilde{u}_0)$ for all $w\in H_0^1(\Omega)$ with $\|w\|=\widetilde{r}$;

$\mathrm{(ii)}$ For any $v\in H_0^1(\Omega)\backslash\{0\}$, there exists $\beta^{*}>0$ such that $\|\widetilde{u}_0+\beta^{*}v\|>\widetilde{r}$ and
$\widetilde{J}(\widetilde{u}_0+\beta^{*}v)\leq\widetilde{J}(\widetilde{u}_0)$.
\end{lemma}

\begin{lemma}\label{5.4}
Assume that $N=3,5$, $(\lambda,\mu,\theta)\in A_3\cup A_4$ and
\begin{align*}
c<\widetilde{c}_{K}+\frac{1}{N}\mu^{-\frac{N-2}{2}}S^{\frac{N}{2}},
\end{align*}
then $\widetilde{J}(u)$ satisfies the $(PS)_c$ condition.
\end{lemma}

Without loss of generality, assume that there exists $\varrho >0$ such that $B_{2\varrho}(0)\subset\Omega$ and
\begin{align}\label{5-eq1}
\widetilde{m}_{\varrho}:=\inf_{x\in B_{\varrho}(0)}\widetilde{u}_0(x)\geq\frac{1}{2}\widetilde{u}_0(0).
\end{align}
Let $\varphi(x)\in C_0^{\infty}(\Omega)$ be a radial function such that $\varphi(x)\equiv1$ for $0\leq|x|\leq\varrho$, $0\leq\varphi(x)\leq1$ for $\varrho\leq|x|\leq2\varrho$ and $\varphi(x)\equiv0$ for $|x|\geq2\varrho$.
Define
$$U_{\varepsilon}(x)=C_N\left(\frac{\varepsilon}{\varepsilon^2+|x|^2}\right)^{\frac{N-2}{2}}, \ \ x\in \mathbb{R}^N,$$
and
$$\phi_{\varepsilon}=\varphi(x)U_{\varepsilon}(x),$$
where $C_N=[N(N-2)]^{\frac{N-2}{4}}$.
Then we have the following results.

\begin{lemma}\label{5.5}
Assume that $N=3,5$. Then as $\varepsilon\rightarrow 0^+$,
\begin{eqnarray*}
\int_{\Omega}|\nabla \phi_{\varepsilon}|^2\mathrm{d}x=S^{\frac{N}{2}}+O(\varepsilon^{N-2}),
\end{eqnarray*}
\begin{eqnarray*}
\int_{\Omega}|\phi_{\varepsilon}|^{2^*}\mathrm{d}x=S^{\frac{N}{2}}+O(\varepsilon^N),
\end{eqnarray*}
\begin{eqnarray*}
\int_{\Omega}|\phi_{\varepsilon}|^{2^*-1}\mathrm{d}x=O(\varepsilon^{\frac{N-2}{2}}),
\end{eqnarray*}
\begin{eqnarray*}
\int_{\Omega}|\phi_{\varepsilon}|\mathrm{d}x=O(\varepsilon^{\frac{N-2}{2}}),
\end{eqnarray*}
and
\begin{eqnarray*}
\int_{\Omega}|\phi_{\varepsilon}|^2\mathrm{d}x=
\begin{cases}
O(\varepsilon), \ &if\ N=3,\\
d\varepsilon^2+O(\varepsilon^3), \ &if\ N=5,
\end{cases}
\end{eqnarray*}
where $d$ is a positive constant.
\end{lemma}
\begin{proof}
The proof can be found in \cite{Brezis1983}, [\cite{Capozzi1985}, Lemma 2.1] and \cite{Dengyinbin}.
\end{proof}

\begin{lemma}\label{5.6}
Assume that $N=3,5$. Then as $\varepsilon\rightarrow 0^+$,
\begin{align*}
\int_{\Omega}|\phi_{\varepsilon}|^{2+2\delta}\mathrm{d}x=O(\varepsilon^{2-(N-2)\delta}),
\end{align*}
and
\begin{align*}
\int_{\Omega}\widetilde{u}_0|\phi_{\varepsilon}|^{2^{*}-1}\mathrm{d}x\geq \frac{1}{2}C_N^{2^*-1}\widetilde{u}_0(0)\varepsilon^{\frac{N-2}{2}}\int_{\mathbb{R}^N}\frac{1}{(1+|y|^2)^{\frac{N+2}{2}}}\mathrm{d}y
+ O(\varepsilon^{\frac{N+2}{2}}),
\end{align*}
where $\max\{0,\frac{4-N}{2(N-2)}\}<\delta<\frac{6-N}{2(N-2)}$.
\end{lemma}
\begin{proof}
The proof follows from the strategies in Lemma \ref{estimate3}, so we omit the details here.
\end{proof}

\begin{lemma}\label{5.7}
Assume that $N=3,5$ and $\beta>0$. Then we have
\begin{align}\label{ineq5-1}
&\int_{\Omega}(\widetilde{u}_0+\beta \phi_{\varepsilon})^2\log(\widetilde{u}_0+\beta \phi_{\varepsilon})^2\mathrm{d}x
-\int_{\Omega}\widetilde{u}_0^2\log \widetilde{u}_0^2\mathrm{d}x
-\int_{\Omega}2\beta \widetilde{u}_0 \phi_{\varepsilon}\log \widetilde{u}_0^2\mathrm{d}x\nonumber\\
\leq&~\frac{2}{e\delta}\beta^2\int_{\Omega}|\widetilde{u}_0|^{2\delta}|\phi_{\varepsilon}|^2\mathrm{d}x
+\frac{2}{e\delta}\beta^{2+2\delta}\int_{\Omega}|\phi_{\varepsilon}|^{2+2\delta}\mathrm{d}x
+2\beta\int_{\Omega}\widetilde{u}_0\phi_{\varepsilon}\mathrm{d}x+6\beta^2\int_{\Omega}|\phi_{\varepsilon}|^2\mathrm{d}x,
\end{align}
and
\begin{align}\label{ineq5-2}
\int_{\Omega}|\widetilde{u}_0+\beta \phi_{\varepsilon}|^{2^*}\mathrm{d}x
\geq&\int_{\Omega}|\widetilde{u}_0|^{2^*}\mathrm{d}x+\beta^{2^*}\int_{\Omega}|\phi_{\varepsilon}|^{2^*}\mathrm{d}x
+2^*\beta\int_{\Omega}|\widetilde{u}_0|^{2^*-2}\widetilde{u}_0\phi_{\varepsilon}\mathrm{d}x\nonumber\\
&+2^*\beta^{2^*-1}\int_{\Omega}\widetilde{u}_0|\phi_{\varepsilon}|^{2^*-1}\mathrm{d}x,
\end{align}
where $\max\{0,\frac{4-N}{2(N-2)}\}<\delta<\frac{6-N}{2(N-2)}$.
\end{lemma}

\begin{proof}
The proof of \eqref{ineq5-1} is similar to that of Lemma \ref{estimate2}. Inequality \eqref{ineq5-2} can be calculated directly.
\end{proof}

Define
$$\widetilde{\Gamma}:=\left\{\gamma\in C([0,1],H_0^1(\Omega)):\gamma(0)=\widetilde{u}_0,\widetilde{J}(\gamma(1))\leq \widetilde{J}(\widetilde{u}_0)\right\},$$
and
$$\widetilde{c}_{M}:=\inf_{\gamma\in \widetilde{\Gamma}}\max_{t\in[0,1]}\widetilde{J}(\gamma(t)).$$
Then the inequality $\widetilde{c}_{M}<\widetilde{c}_{K}+\frac{1}{N}\mu^{-\frac{N-2}{2}}S^{\frac{N}{2}}$ is still valid
for $N=3,5$.

\begin{lemma}\label{5.8}
Assume that $N=3,5$ and $(\lambda,\mu,\theta)\in A_3\cup A_4$. Then we have
$$\widetilde{c}_{M}<\widetilde{c}_{K}+\frac{1}{N}\mu^{-\frac{N-2}{2}}S^{\frac{N}{2}}.$$
\end{lemma}

\begin{proof}
From Lemma 2.1 in \cite{Dengyinbin}, we know that there exists $\alpha_0>0$ such that $\widetilde{J}(u)\geq\alpha_0$
for all $\|u\|=\widetilde{\rho}$, where $\widetilde{\rho}$ is defined in \eqref{eq5-1}. For $\phi_{\varepsilon}$,
we may take $\beta^{*}$ large enough in Lemma \ref{5.3}  such that $\|\widetilde{u}_0+\beta^{*}\phi_{\varepsilon}\|>\widetilde{\rho}$.
Then by the definition of $\widetilde{c}_{M}$, one has
\begin{align}\label{ineq5-3}
0<\widetilde{c}_{M}
\leq\max_{t\in [0,1]}\widetilde{J}(\widetilde{u}_0+t\beta^{*}\phi_{\varepsilon})
\leq\sup_{\beta\geq0}\widetilde{J}(\widetilde{u}_0+\beta\phi_{\varepsilon}).
\end{align}
Obviously, it is sufficient to prove that
$$\sup_{\beta\geq0}\widetilde{J}(\widetilde{u}_0+\beta\phi_{\varepsilon})<\widetilde{c}_{K}+\frac{1}{N}\mu^{-\frac{N-2}{2}}S^{\frac{N}{2}}.$$
Similar to the proof of Lemma \ref{estimate4}, by using Lemmas \ref{5.5}, \ref{5.6} and \ref{5.7}, we can deduce that
\begin{align*}
\sup_{\beta\geq0}\widetilde{J}(\widetilde{u}_0+\beta\phi_{\varepsilon})<\widetilde{c}_{K}+\frac{1}{N}\mu^{-\frac{N-2}{2}}S^{\frac{N}{2}}.
\end{align*}
This completes the proof.
\end{proof}

{\bf Proof of Theorem \ref{th5.1}.}
Assume that $N=3,5$ and $(\lambda, \mu, \theta)\in A_3\cup A_4$.
By Lemmas \ref{5.3}, \ref{5.4} and \ref{5.8} and the Mountain Pass Theorem, we know that
there exists $u\in H_0^1(\Omega)$ such that
\begin{align}\label{eq5-2}
\widetilde{J}(u)=\widetilde{c}_{M}>0, \ \ \ \ \ \ \ \  \widetilde{J}'(u)=0.
\end{align}
From \eqref{eq5-2}, we have $\langle J'(u), u^{-}\rangle=\int_{\Omega}|\nabla u^{-}|^2\mathrm{d}x=0$,
which means that $u\geq0$ in $\Omega$. Therefore, $u$ is a nonnegative mountain pass solution to problem \eqref{P2} with $N=3,5$.
Then by a similar argument to that used in the proof of Theorem \ref{th1.1}, we can show that $u>0$
and $u\in C^2(\Omega)$. This completes the proof.

\end{document}